\documentclass{amsart}
\usepackage{amsfonts}
\usepackage{amsmath}
\usepackage{amssymb}
\usepackage{graphicx}
\newtheorem{theorem}{Theorem}
\theoremstyle{plain}

\newtheorem{corollary}{Corollary}

\newtheorem{definition}{Definition}

\newtheorem{proposition}{Proposition}
\newtheorem{remark}{Remark}

\numberwithin{equation}{section}

\newcommand{\R}{\mathbb{R}}
\newcommand{\RP}{\mathbb{RP}}

\renewcommand{\H}{\mathbb{H}}
\newcommand{\E}{\mathbb{E}}
\renewcommand{\S}{\mathbb{S}}
\newcommand{\G}{\mathbb{G}}
\newcommand{\bigO}{O}

\newcommand{\set}[1]{\left\{ #1 \right\}}

\newcommand{\dij}{d_{ij}}
\newcommand{\dji}{d_{ji}}

\DeclareMathOperator{\vecspan}{Span}
\usepackage{hyperref}

\begin{document}
\title[Duality structures and conformal variations of surfaces]
{Duality structures and discrete conformal variations of piecewise constant curvature surfaces}
\author{David Glickenstein}
\thanks{Partially funded by NSF grant DMS 0748283.}
\address{Mathematics Department\\ University of Arizona \\Tucson, AZ 85721}
\email[David Glickenstein]{glickenstein@math.arizona.edu}
\author{Joseph Thomas}
\email[Joseph Thomas]{jthomas@math.arizona.edu}

\begin{abstract}
  A piecewise constant curvature manifold is a triangulated
  manifold that is assigned a geometry by specifying lengths of
  edges and stipulating that for a chosen background
    geometry (Euclidean, hyperbolic, or spherical), each
  simplex has an isometric embedding into the background
  geometry with the chosen edge lengths. Additional structure
  is defined either by giving a geometric structure to the 
  Poincar\'e dual of the triangulation or by assigning
  a discrete metric, a way of assigning length to oriented edges.
  This notion leads to a notion of discrete conformal structure, 
  generalizing the discrete conformal structures based on circle packings
  and their generalizations studied by Thurston and others.
  We define and analyze conformal variations of 
  piecewise constant curvature 2-manifolds, giving
  particular attention to the variation of angles. We give formulas
  for the derivatives of angles in each background geometry, which
  yield formulas for the derivatives of curvatures. Our formulas
  allow us to identify particular curvature functionals associated
  with conformal variations. Finally, we provide a complete classification
  of discrete conformal structures in each of the background geometries.
\end{abstract}
\maketitle


\section{Introduction}

A triangulation of a manifold can be given a geometric structure by
assigning compatible geometric structures to its component
simplices. One of the easiest ways of doing this is to assign constant
curvature geometries to the simplices, as these simplices are uniquely
determined by their edge lengths.  Such a structure
gives a finitely parametrized set of geometric structures on a closed
manifold.

In Thurston's formulation of the discrete Riemann mapping problem (see
\cite{Ste}) as well as in applied methods such as discrete exterior
calculus (see, e.g., \cite{DHLM}, \cite{DAOD}), it is important to not
only have a piecewise constant curvature metric assigned to simplices,
but also to give a structure to the Poincar\'e dual of the
triangulation.  Such structures arise naturally as incircle duals in
Thurston's formulation of circle packings and as circumcentric duals
in discrete exterior calculus. For piecewise Euclidean surfaces and
3-manifolds, in \cite{G3} and \cite{G5} the first author gives an
axiomatic treatment of geometric duality structures that have
orthogonal intersections with the primal simplices, and also relates
these to discrete conformal variations.

The goal of the present work is to make precise the parametrization of
duality structures by partial edge lengths (giving a discrete analogue
of a Riemannian metric), define the general form of discrete conformal
structures based on an axiomatic development related to conformal
variation of angle, and derive a local classification of such
structures. The relationship between duality structures and discrete
metrics requires some understanding of possible geometric centers for
triangles, leading to the definition of the span of a triangle as the
space of possible geometric centers.  The axiomatic development of
conformal structure follows that in \cite{G5} for piecewise Euclidean
surfaces, while the construction in piecewise hyperbolic and spherical
surfaces is new. The general formulas for angle and curvature
variation of piecewise hyperbolic and spherical surfaces is new
(however, see the parallel work in \cite{ZGZLYG}), generalizing circle
packings and other discrete conformal structures previously studied by
many authors (see Section \ref{sec: previous formulations} for
details).  The local classification of discrete conformal structures,
giving explicit formulas for the structures, is new for each geometry
including Euclidean.

We will begin by making these geometric structures precise, and then give
precise statements of the main results.

\subsection{Geometric structures on triangulations}
In this section, we make precise some geometric structures.
\begin{definition}
  A triangulated manifold $(M,T)$ is a topological manifold $M$
  together with a triangulation $T$ of $M$. A (triangulated)
  \emph{piecewise constant curvature manifold} $(M,T,\ell)$ \emph{with
    background geometry $\G$} is a triangulated manifold $(M,T)$
  together with a function $\ell$ on the edges of the triangulation
  such that each simplex can be embedded in $\G$, a space of constant
  curvature, as a (nondegenerate) simplex with edge lengths determined
  by $\ell$.
  
  When the background geometry is Euclidean ($\G=\E$), hyperbolic
  ($\G=\H$), or spherical ($\G=\S$), we call such a manifold
  \emph{piecewise flat}, \emph{piecewise hyperbolic}, or
  \emph{piecewise spherical}, respectively.
\end{definition}
 
When the background geometry is clear from context, we may omit
it. Note that part of the definition is that the simplices are
nondegenerate; this places inequality restrictions on the possible
edge lengths. For instance, in Euclidean background the restrictions
can be derived from Cayley-Menger determinants.

We will use $V=V(T)$ to denote the vertices in triangulation $T$ and
label them with numbers or letter such as $i\in V$. We will use
$E=E(T)$ to denote edges and label them as a set of vertices
$\{i,j\}\in E$, although most of this work could allow multiple edges
between the same vertices or edges between the same vertex. We will
use $E_+=E_+(T)$ to denote oriented edges and label them with ordered
pairs $(i,j)\in E_+$. Triangles will be denoted as a set of vertices,
such as $\{i,j,k\}$. In a piecewise constant curvature manifold, the
angle at vertex $i$ in a triangle $\{i,j,k\}$ will be denoted
$\gamma_i$. The set of real valued functions on $V$ or $E_+$ will be
denoted by $V^\ast$ and $E_+^\ast$, respectively. We will use $\sigma
< \tau$ to mean that $\sigma$ is a subsimplex of $\tau$.

\subsubsection{Duality structures}

The idea of a duality structure is that, in addition to the metric structure
of a piecewise constant curvature manifold, we can put a geometric structure
on the Poincar\'e dual cell complex by introducing geometric centers
for pieces of the dual complex. Motivated by the Euclidean background case,
we see that these geometric centers do not have to be 
constrained to the simplex, but its affine span. In the more general 
constant curvature case, we will need an analogue of the
affine span that defines the space of possible simplex centers.

Since a piecewise constant curvature manifold is subdivided by
simplices that can be embedded into the space $\G$, each simplex
$\sigma^k$ has a span defined as follows. First we need to define the 
underlying space of the span in each geometry.
\begin{definition}
Given a constant curvature geometry $\G$, we define $\hat{\G}$ as follows:
\begin{itemize}
\item If $\G=\E^n$ then we take $\hat{\G}$ to be the underlying space $\R^n$.
\item If $\G=\H^n$ then we take  $\hat{\G}$ to be the entire space of the Klein model, also described as the 
extended hyperbolic plane in \cite{ChoKim}. Note that in this case, $\H^n \subset \hat{\H}^n$.
\item If $\G=\S^n$ then we take $\hat{\G}$ to be the quotient $\RP^n$ of the sphere.
\end{itemize} 
\end{definition}
We note the following easy facts about $\hat{\G}$.
\begin{itemize}
\item In each case, the isometry group of $\G$ acts on $\hat{\G}$.
\item In each case, there is a notion of orthogonality between two
  vectors, induced from the Euclidean dot product in the cases of
  Euclidean space and the sphere, and the Lorentzian bilinear product
  using the hyperboloid model of hyperbolic space and projecting to
  the Klein model space.
\item In each case, any two points in $\hat{\G}$ can be connected by a line.
\end{itemize}
%

In what follows, we will assume that any simplex modeled on geometry $\G$ can be isometrically
embedded into $\hat{\G}$, and that embedding is unique up to isometry of $\G$. Note that, in the case
of spherical geometry, the fact that a simplex embeds is a restriction on how big it can be. We are now
ready to define the span.

\begin{definition}

  Given a simplex $\sigma^k$ and an isometric embedding $\phi:
  \sigma^k \to \hat{\G}$, the \emph{span of $\sigma^k$ under $\phi$},
  denoted $S_\phi\sigma^k$, is the set
  \begin{align*}
    S_\phi\sigma^k = \bigcup_{\substack{p,q \in \phi(\sigma^k) \\ p \neq q}}{L_{p,q}}
  \end{align*}
  where $L_{p,q} \subset \hat{\G}$ is the line through the points $p$
  and $q$. 

  The \emph{span of $\sigma^k$}, denoted $S\sigma^k$, is the quotient
  space obtained from the disjoint union $\bigsqcup_\phi{S_\phi\sigma^k}$ by identifying
  each pair of summands $S_\phi\sigma^k$ and $S_\rho \sigma^k$ by an
  isometry of $\G$ that agrees with $\rho\circ\phi^{-1}$.
\end{definition}

We remark that our definition is analogous to the definition of
affine span in polytope theory (c.f., \cite{millerpak}). In both
definitions, the span is viewed as a (geodesic) hyperplane tangent to
the simplex/polytope-face as it sits in the ambient geometry.


The span has the property that for any points $x\in \sigma\subset
S\sigma$ and $y\in S\sigma$, there is a unique line between $x$ and
$y$ in $S\sigma \cong \hat{\G}$. The span also
has the property that if $\sigma <\sigma'$ then there is a natural way
in which $S \sigma \subset S\sigma'$.

\begin{definition}
  Suppose $(M,T,\ell)$ is a piecewise constant curvature manifold with
  background geometry $\G$.

  A \emph{duality structure} for $(M,T)$ is a choice of one point
  $C[\sigma] \in S\sigma$ from each simplex $\sigma^k$ of $T$, subject
  to:
  \begin{quote}
    If $\sigma^\ell < \sigma^k$ then for any simplex
    $\tau=\{C[\sigma^\ell], C[\sigma^{\ell+1}], \ldots,
    C[\sigma^k]\}$, we have that $S\tau$ is orthogonal to
    $S\sigma^\ell$ intersecting only at $C[\sigma^\ell]$.
  \end{quote}

  We say a duality structure is \emph{proper} if it has Euclidean or
  spherical background or has hyperbolic background and the center of
  each edge is in $\H$.
\end{definition}
Notice that in the case of spherical background, the centers lie in $\RP^n$ and so
correspond to two points in $\S^n$. We will often consider the span as $\S^n$ with pairs of
points instead of $\RP^n$. Proper duality structures are ones such that edge centers are 
determined by signed distances from the vertices, as determined by the partial edge lengths
in the next section.

In general, we will denote the center of edge $\{i,j\}$ by $c_{ij}$ and the center of 
triangle $\{i,j,k\}$ by $c_{ijk}$. These centers determine edge heights.
\begin{definition}
	Given a proper duality structure on a triangle $\{i,j,k\}$, each edge $\{i,j\}$ has a corresponding
	edge height $h_{ij}$ determined by one of the following:
	\begin{itemize}
		\item If the center $c_{ijk}$ is in the same half plane determined by the span $S\{i,j\}$
			as the simplex $\{i,j,k\}$ is, $h_{ij}$ is the distance between $c_{ij}$ and $c_{ijk}$.
		\item If the center $c_{ijk}$ is not in the same half plane determined by the span $S\{i,j\}$
					as the simplex $\{i,j,k\}$ is, $h_{ij}$ is the negative of the distance between $c_{ij}$ and $c_{ijk}$.
		\item If the center $c_{ijk}$ is in $\hat{\H}$ but not in $\H$, then the height is the distance from $c_{ij}$ to $c_{ijk}^\perp$ (see Section \ref{section:hyperbolic basics}) with the same sign convention.
	\end{itemize}
\end{definition}

\subsubsection{Discrete metric structure}
The definition of duality structure requires choosing centers. For a more
explicit parametrization, we will
try to adjust the metric structure $\ell$ in some way to ensure a duality 
structure. This is the role of metrics and  pre-metrics.

The notion of a  pre-metric is to reassign parts of the length function to 
the vertices. This is motivated partly by the definition of Riemannian metrics
as tensor valued functions of the points of a manifold.

\begin{definition} \label{def: pre-metric}

  Let $\left( M,T\right)$ be a triangulated manifold. A
  \emph{pre-metric} is an element $d\in E_{+}\left( T\right) ^{\ast}$
  such that $\left( M,T,\ell\right) $ is a piecewise constant
  curvature manifold with background geometry $\G$ for the assignment
  $\ell_{ij}=d_{ij}+d_{ji}$ for every edge $\left\{ i,j\right\}.$

\end{definition}

The $d_{ij}$ are sometimes called partial edge lengths, since one
considers the edge $\{i,j\}$ divided into two partial edges of length
$d_{ij}$ and $d_{ji}$. If the partial edge lengths are nonnegative,
there is a point on the edge that is distance $d_{ij}$ from vertex $i$
and distance $d_{ji}$ from vertex $j$, and this point is called the
edge center. Note that if one of the partial edge lengths is negative,
there is an interpretation in terms of signed distance, and there is
still a center, this time on the span of the edge.

We would like to restrict pre-metrics to those that generate
geometries on the Poincar\'e dual structure such that dual and primal
cells intersect orthogonally.  If one considers the point $c_{ij}$ on
the span of an edge $\{i,j\}$ that is distance $d_{ij}$ from vertex
$i$ and $d_{ji}$ from vertex $j$ (distance can be considered with sign 
so one partial edge length can be negative), 
a center is determined. One can consider the plane
orthogonal to the span S$\{i,j\}$ through $c_{ij}$, and use the intersections of these planes to
construct more centers (e.g., if the planes of the three edges of a
triangle intersect at a point then we use that point as the center of
the triangle). This construction is explained in detail for Euclidean
background in \cite{G3}. We wish to characterize which
conditions on the pre-metrics guarantee that these centers exist and
give a duality structure.  We call these metrics, and the actual
motivations for the following definitions are characterization
theorems given later. The main advantage of metrics over duality
structures is that the metrics entirely parametrize the geometry, and
so the space of metrics is relatively easy to describe.

\begin{definition}
  \label{def: metric}
  A \emph{discrete metric}, or \emph{metric}, on $\left(  M,T\right)$ with background geometry $\G$ 
  is a  pre-metric $d$  such that for
  every triangle $\left\{  i,j,k\right\}  $ in $T$,
  \begin{align}
    d_{ij}^{2}+d_{jk}^{2}+d_{ki}^{2}&=d_{ji}^{2}+d_{kj}^{2}+d_{ik}^{2} &\text{if $\G=\E$,}
    \label{d-euclidean condition}\\
     \cosh( d_{ij} ) \cosh( d_{jk} ) \cosh( d_{ki} ) &= 
     \cosh( d_{ji} ) \cosh( d_{kj} ) \cosh( d_{ik} ) &\text{if $\G=\H$,}
     \label{d-hyperbolic condition}\\
     \cos( d_{ij} ) \cos( d_{jk} ) \cos( d_{ki} ) &= 
     \cos( d_{ji} ) \cos( d_{kj} ) \cos( d_{ik} ) &\text{if $\G=\S$.}
     \label{d-spherical condition}
    \end{align}
  
  A \emph{piecewise constant curvature, metrized manifold} $\left( M,T,d\right)$ with background
  geometry $\G$ is a triangulated manifold $\left(  M,T\right)$ together with a metric $d$. 
  We denote the space of all metrics 
  with background geometry $\G$ on a given
  triangulated manifold $\left( M,T\right)$ by
  $\mathfrak{met_\G}\left(M,T\right)$.
\end{definition}
Note that the space of metrics $\mathfrak{met_\G}\left(M,T\right)$ on a finite triangulation is 
determined as a subset of $\R^{|E_+|}$ by a number of equalities of the form above (one for each
triangle) and a number of inequalities (to ensure the simplices are nondegenerate).

\subsubsection{Discrete conformal structure}

A discrete conformal structure is a particular way of determining the
metric from information assigned to points (vertices). It is partly
motivated by this characterization of conformal change of a Riemannian
metric, and also by Thurston's formulation of conformal circle packing
structure. A general formulation for Euclidean background is described
in \cite{G5}, and there are a number of formulations of specific cases
of analogous structures in hyperbolic and spherical backgrounds (see
Section \ref{sec: previous formulations}).

Based on Propositions \ref{prop:euclidean compatibility},
\ref{prop:hyperbolic compatibility}, and \ref{prop:spherical
  compatibility}, if we suppose that the pre-metric is determined by
weights on the vertex endpoints, there is a restriction that ensures
that the resulting pre-metric is actually a discrete metric, i.e., it
determines a duality structure. In addition, we want conformal
structures to have nice formulas for angle variations. This motivates
the following definition.

\begin{definition} \label{def:conformal structure}
A \emph{discrete conformal structure} $\mathcal{C}\left(  M,T,U\right)  $ on a
triangulated manifold $\left(  M,T\right)$ with background geometry $\G$
 on an open set $U\subset V\left(
T\right)  ^{\ast}$ is a smooth map%
\[
  \mathcal{C}\left(  M,T,U\right):U \rightarrow \mathfrak{met_\G}\left(M,T\right)
\]
such that if $d=\mathcal{C}\left(  M,T,U\right)  \left[  f\right]  $ then for
each $\left(  i,j\right)  \in E_{+}(T)$ and $k\in V(T)$,
\begin{align}
  \frac{\partial\ell_{ij}}{\partial f_{i}} &=d_{ij} & \text{if $\G=\E$,}\\
  \frac{\partial\ell_{ij}}{\partial f_{i}}&=\tanh d_{ij} &\text{if $\G=\H$,}\\
  \frac{\partial\ell_{ij}}{\partial f_{i}}&=\tan d_{ij} &\text{if $\G=\S$,}
\end{align}
and
\[
  \frac{\partial d_{ij}}{\partial f_{k}}=0
\]
if $k\neq i$ and $k\neq j$.

A \emph{conformal variation} of a metric $d=\mathcal{C}\left(  M,T,U\right)[f]$ is the change of the 
metric in the conformal class as $f$ changes, 
and is determined by derivatives such as $\partial d_{ij}/\partial f_i$.  
\end{definition}
We have chosen the parameter $f$ so that the variation formulas above are as simple as possible. However, we will
sometimes choose to parametrize the structures differently (see Theorem \ref{thm:functional}). Also note that with conformal variations, the choice
of the set $U$ is not particularly important; we only need the existence of a neighborhood around
any point in $U$.


\subsection{Main theorems}
In this paper, we study the relationships between duality structures, discrete metrics,
 and conformal variations. The main new contributions are the
following: (1) a characterization of duality structures in hyperbolic
and spherical backgrounds, generalizing the notion of length structures arising from circles
with given radii and inversive distances, (2) calculation of the conformal variation of angles
in a triangle for hyperbolic and spherical backgrounds together with determining a functional making the
curvature variational, and (3) a classification theorem for
discrete conformal variations of Euclidean, hyperbolic, and spherical triangles, 
including the formulation of the notion of discrete conformal variations from basic principles.
\subsubsection{Equivalence of duality and metric structures}
The following theorem characterizes duality structures on surfaces in each of the constant curvature backgrounds.
\begin{theorem} \label{thm:duality equals metric}
  Let $(M,T,\ell)$ be a piecewise constant curvature 2-manifold. There is a one-to-one correspondence between 
  proper duality structures on $(M,T,\ell)$ and discrete metric structures on $(M,T,\ell)$. 
\end{theorem}

This theorem follows from Propositions \ref{prop:euclidean compatibility}, \ref{prop:hyperbolic compatibility},
and \ref{prop:spherical compatibility}.
\subsubsection{Discrete conformal variations of angle}
The following theorem gives the variation of angle formulas. The Euclidean result is in \cite{G5}, and the
hyperbolic and spherical results are new (compare \cite{ZGZLYG}).
\begin{theorem}\label{thm: angle variation}
For any conformal variation of a metric $d=\mathcal{C}\left(  M,T,U\right)[f]$ with background
geometry $\G$ of a surface $M^2$, 
we have for any edge $\{i,j\}$ the following formulas.
\begin{itemize}
\item In Euclidean background,
\begin{align}
  \frac{\partial \gamma_i}{\partial f_j} &= \frac{h_{ij}}{\ell_{ij}} \\
  \frac{\partial \gamma_i}{\partial f_i} &= -\frac{h_{ij}}{\ell_{ij}}-\frac{h_{ik}}{\ell_{ik}}.
\end{align}
\item In hyperbolic background,
  \begin{align}
    \frac{\partial \gamma_i}{\partial f_j} 
    &= \frac{1}{\cosh d_{ji}}\frac{\tanh^\beta h_{ij}}{\sinh \ell_{ij}} 
    \label{eqn:angle variation ij} \\ 
    \frac{\partial \gamma_i}{\partial f_i} 
    &= -\frac{1}{\cosh d_{ji}}\frac{\tanh^\beta h_{ij}}{\tanh \ell_{ij}} 
    -\frac{1}{\cosh d_{ki}}\frac{\tanh^\beta h_{ik}}{\tanh \ell_{ik}} \label{eqn:angle variation ii}
    \end{align}
  where $\beta$ is 1 if $c_{ijk}$ is timelike and -1 if $c_{ijk}$ is
  spacelike.
\item In spherical background,
  \begin{align}
    \frac{\partial \gamma_i}{\partial f_j} 
    &= \frac{1}{\cos d_{ji}}\frac{\tan h_{ij}}{\sin \ell_{ij}} 
    \label{eqn:angle variation ij spherical} \\ 
    \frac{\partial \gamma_i}{\partial f_i} 
     &=-\frac{1}{\cos d_{ji}}\frac{\tan h_{ij}}{\tan \ell_{ij}} 
        -\frac{1}{\cos d_{ki}}\frac{\tan h_{ik}}{\tan \ell_{ik}}\label{eqn:angle variation ii spherical}
  \end{align}
\end{itemize}
\end{theorem}
This theorem follows from Theorems \ref{thm: Euclidean variation angle}, \ref{thm: hyperbolic angle variation},
and \ref{thm: spherical angle variation} together with Propositions \ref{prop:hyperbolic area deriv} and 
\ref{prop: sphere area var}.

It turns out that although the variables $f$ for the conformal variations are quite natural,
a change of variables gives that the curvatures are the gradient of a functional, where the curvatures are defined 
as 
\[
  K_i = 2\pi - \sum_{\{i,j,k\}} \gamma_i
\]
for each vertex $i$, where the sum is over all triangles containing $i$.

\begin{theorem} \label{thm:functional}
  Consider a piecewise constant curvature, metrized 2-manifold 
  $(M,T,d)$, where $d=d(f)$ is determined by a conformal structure.
  There is a change of variables $u=u(f)$ such that 
  \[
    \frac{\partial \gamma_i}{\partial u_j}=\frac{\partial \gamma_j}{\partial u_i}
  \]
  and hence if we fix a $\bar{u}$ there is a functional
  \[
  	F=2\pi\sum_{i\in V} u_i - \sum_{\{i,j,k\}}\int_{\bar{u}}^{u}(\gamma_i du_i+\gamma_j du_j+\gamma_k du_k)
  \]
  with the property that 
  \[
  \frac{\partial F}{\partial u_i} = K_i.
  \]
  Furthermore, if all $d_{ij}>0$ and $h_{ij}>0$ and then this function is strictly convex if $\G=\H$ and
  weakly convex (strictly convex except for scaling) if $\G=\E$.
\end{theorem}
This theorem follows from Theorems \ref{thm: Euclidean functional}, \ref{thm:hyp functional},
and \ref{thm:sphere functional}.

\subsubsection{Classification of discrete conformal structures}

The following theorems classify discrete conformal variations in each of the constant curvature backgrounds. The results
are new for all background geometries.
\begin{theorem} \label{thm:conformal classification}
  Let $\mathcal{C}\left(  M,T,U\right)$ be a discrete conformal class with background geometry $\G$ on a surface $M$. Then there exist $\alpha\in\mathbb{R}^{\left\vert V\right\vert }$
  and $\eta\in\mathbb{R}^{\left\vert E\right\vert }$ such that the conformal
  structure can be written as
  \[
    d_{ij}=\frac{\alpha_{i}e^{2f_{i}}+\eta_{ij}e^{f_{i}+f_{j}}}{\ell_{ij}}
  \]
  with
  \[
    \ell_{ij}^{2}=\alpha_{i}e^{2f_{i}}+\alpha_{j}e^{2f_{j}}+2\eta_{ij} e^{f_{i}+f_{j}}.
  \]
  if $\G=\E$,
  \begin{align*}
    \tanh d_{ij}  & = \frac{\alpha_{i}e^{2f_{i}}}{\sinh
				\ell_{ij}}\sqrt{\frac{1+\alpha_{j}e^{2f_{j}}}{1+\alpha_{i}e^{2f_{i}}}}
				+\frac{\eta_{ij}e^{f_{i}+f_{j}}}{\sinh\ell_{ij}}
  \end{align*}

with
  \[
		\cosh\ell_{ij}=\sqrt{\left(  1+\alpha_{i}e^{2f_{i}}\right)  
		\left( 1+\alpha_{j}e^{2f_{j}} \right)  }+\eta_{ij}e^{f_{i}+f_{j}}.
  \]
 if $\G=\H$, or
  \begin{align*}
    \tan d_{ij} & =
    \frac{\alpha_{i}e^{2f_{i}}}{\sin\ell_{ij}}\sqrt{\frac{1-\alpha_{j}e^{2f_{j}}}{1-\alpha_{i}e^{2f_{i}}}}
    +\frac{\eta_{ij}e^{f_{i}+f_{j}}}{\sin\ell_{ij}}
  \end{align*}
with
\begin{align*}
  \cos\ell_{ij}=\sqrt{\left(1-\alpha_{i}e^{2f_{i}}\right) \left(1-\alpha_{j}e^{2f_{j}}\right)}
  -\eta_{ij}e^{f_{i}+f_{j}}.
\end{align*}
if $\G=\S$.
\end{theorem}

This theorem is proven in each case in Sections \ref{section: eucl conf classify}, \ref{section: hyp conf classify},
and \ref{sec: sphere}.

In light of Theorem \ref{thm:conformal classification}, one can also calculate angle variations from 
Theorem \ref{thm: angle variation} based on the conformal structures determined by 
$\alpha$ and $\eta$. These conformal structures are sometimes referred to as $\mathcal{C}_{\alpha, \eta}$
(see, e.g., \cite{G6}).

\subsection{Comparison with previous formulations} \label{sec: previous formulations}
In this section we briefly compare our parametrizations with other parametrizations
of certain discrete conformal structures. The formulation in this paper
unifies the previous work into a single formula
for each background geometry and generalizes some of these. Independently,
\cite{ZGZLYG} derived a formula for the variation of angle that is essentially
the same as ours, though we express it and prove it in a different way. We note that
the Euclidean background case was treated in \cite{G5}, which also describes
the relationship of the general case to previous formulations. 

The first formulation of the circle packing conformal structure (corresponding, in our
notation, to $\alpha_i=1$ and $\eta_{ij}=1$ for all vertices and edges) is 
in Thurston's work \cite{Thurs}. Many of the relevant calculations
are followed through in \cite{MR}, and the first variational 
formulation is due to Colin de Verdi\`{e}re in \cite{CdV}. In each of these
cases, the Euclidean and hyperbolic cases were treated, and the conformal structures
were either circles with given intersection angles between $0$ and $\pi/2$
(corresponding, in our notation, to $\alpha_i=1$ and $0 \leq \eta_{ij}\leq 1$ for all vertices and edges). 
Additional work was done by Chow-Luo in \cite{CL}. 
The case of circles with fixed inversive distances (corresponding, in our
notation, to $\alpha_i=1$ and $|\eta_{ij}|\geq 1$ for all vertices and edges) was
introduced by Bowers and Stephenson \cite{BoSte} and the variational perspective
was pursued by Guo in \cite{Guo} (this was anticipated by Springborn's work on 
volumes of hyperideal simplices in \cite{Spr}). 

The multiplicative conformal structure (corresponding, in our
notation, to $\alpha_i=0$ for all vertices)
 was apparently first suggested in \cite{RW}, but most of the
mathematical ideas arose in work of Luo \cite{Luo1} and Springborn-Schrader-Pinkall \cite{SSP} in the
Euclidean case. Generalizing to the hyperbolic case was not obvious, but
work in this direction first appeared in work by Bobenko-Pinkall-Springborn \cite{BPS}. 
It is notable that
the proper parametrization variable is not clear in this case, and 
this issue is discussed in Section \ref{subsect:variational hyp}. The 
unified case for Euclidean background is given in \cite{G5} and the 
hyperbolic case was first described in this paper and independently 
in \cite{ZGZLYG}. For more on some of these discrete conformal structures, see the books \cite{Ste}, \cite{DGL}, and \cite{ZG}.

Explicit calculation of the variation of angle coefficients in the Euclidean 
circle packing case
is due to Z. He \cite{He}, and followed by the first author in \cite{G5}. The
coefficients are closely related to the discrete Laplacians found in \cite{Duf}, \cite{Chun},
 \cite{PP}, \cite{Dub}, \cite{BS2}
\cite{G4}, \cite{DHLM}, \cite{HPW}, \cite{WBHZG}, \cite{WMKG}, and many other places.

There are close connections between these variational viewpoints and hyperbolic
volumes, as evidenced by work of Br\"agger \cite{Bra}, Rivin \cite{Riv}, Garret \cite{Gar}, Leibon \cite{Lei}, Bobenko-Springborn \cite{BS1}, Springborn \cite{Spr},  Springborn-Schröder-Pinkall \cite{SSP}, and Bobenko-Pinkall-Springborn \cite{BPS}, Fillastre-Izmestiev \cite{FI}, and Zhang et. al. \cite{ZGZLYG}.

Some of this work was generalized to discrete conformal structures in three 
dimensions by Cooper-Rivin in \cite{CR} and the first author in \cite{G1} and \cite{G5}.
While the functionals whose variations lead to curvatures in two dimensions are 
possibly related to the log determinant of the Laplacian and surface entropy (see \cite{Lei}), 
in three dimensions
the functional is related to Regge's formulation of the Einstein-Hilbert (total scalar curvature)
functional. See, e.g., \cite{Reg}, \cite{CMS}, \cite{Ham}, \cite{CGY}, \cite{Izm1}, \cite{Izm2}.

\section{Euclidean geometry}
\subsection{Duality structures on Euclidean triangles}
%
%
%


Clearly, the choice of a pre-metric with Euclidean background determines 
the geometry of each triangle $\{i,j,k\}$ and for any
isometric embedding, specifies the triangle's
sides $\set{e_{ij}}$ with lengths $\set{\ell_{ij}}$. Through each finite
edge $e_{ij}$ of the triangle we have a unique line $E_{ij}$, considered 
in $\hat{\E}$.

Suppose we identify $E_{ij}$ with the real number line such that $v_i$
is at the origin and $v_j$ is on the positive $x$ axis. Given these
coordinates, we specify the edge centers $c_{ij}=c_{ji}=C(\{i,j\})$ to be the point $d_{ij}$ on the line.
Note that $d_{ji}$ denotes the distance between $c_{ij}$ and $v_j$, considered with 
a sign determined by which side of $v_j$ in $E_{ij}$ contains $c_{ij}$.

For each edge $\{i,j\}$, there exists a unique line $P_{ij}$ that passes
through $c_{ij}$ and is orthogonal to $E_{ij}$.

In \cite{G3} (Proposition 4), the first author presented a necessary
and sufficient condition on the partial edges to guarantee the three
lines $\set{P_{ij}}$ meet at a single point:

\begin{proposition} \label{prop:euclidean compatibility}
  Suppose $\set{d_{ij}}$ is a Euclidean pre-metric. Then the
  perpendiculars $\set{P_{ij}}$ meet at a single point if and only if
  \begin{align}\label{eqn:dij-reln}
    d_{12}^2 + d_{23}^2 + d_{31}^2 = d_{21}^2 + d_{32}^2 + d_{13}^2.
  \end{align}
\end{proposition}

This motivates the Euclidean case of Definition \ref{def: metric}
and proves the Euclidean case of Theorem \ref{thm:duality equals metric}.
\subsection{Conformal variation of angle}
The conformal structure is defined in such a way as to give the following variational formula.
\begin{theorem} \label{thm: Euclidean variation angle}
Given a conformal structure, we have 
\[
  \frac{\partial \gamma_i}{\partial f_j} = \frac{h_{ij}}{\ell_{ij}}
\]
if $i\neq j$ and
\[
  \frac{\partial \gamma_i}{\partial f_i} = -\frac{h_{ij}}{\ell_{ij}}-\frac{h_{ik}}{\ell_{ik}}.
\]
\end{theorem}
This theorem is proven in \cite{G3}, generalizing the theorems in special cases given in 
\cite{He} and \cite{G1}.  It follows easily (see, e.g., \cite{G5}) that the curvature is variational
with respect to a convex functional.
\begin{theorem} \label{thm: Euclidean functional}
 The partial derivatives of the angles in a triangle are symmetric, i.e.,  
  \[
    \frac{\partial \gamma_i}{\partial f_j}=\frac{\partial \gamma_j}{\partial f_i}
  \]
  and hence if we fix a $\bar{f}$ there is a functional
  \[
  	F=2\pi\sum_{i\in V} f_i - \sum_{\{i,j,k\}}\int_{\bar{f}}^{f}(\gamma_i df_i+\gamma_j df_j+\gamma_k df_k)
  \]
  with the property that 
  \[
  \frac{\partial F}{\partial f_i} = K_i.
  \]
  Furthermore, if all $d_{ij}>0$ and $h_{ij}>0$ and then this function is
  weakly convex (strictly convex except for scaling).
\end{theorem}

\subsection{Characterization of discrete conformal structures} \label{section: eucl conf classify}
In this section we prove the characterization theorem. Recall that the only assumptions are:
\begin{itemize}
	\item The compatibility condition \ref{d-euclidean condition} for the triangle with vertices $v_i$, $v_j$, and $v_k$.
	\item The assumption that $d_{ij}$ depends only on $f_i$ and $f_j$.
\end{itemize}
\begin{proof}[Proof of the Euclidean case of Theorem \ref{thm:conformal classification}]
	We first note that
	\begin{align}
		\frac{\partial\ell_{ij}^{2}}{\partial f_{i}}  & =\ell_{ij}^{2}+d_{ij}
			^{2}-d_{ji}^{2}\label{l deriv}\\
		\frac{\partial\ell_{ij}^{2}}{\partial f_{j}}  & =\ell_{ij}^{2}-\left(
			d_{ij}^{2}-d_{ji}^{2}\right)  \nonumber
	\end{align}
	and that
	\[
		\frac{\partial^{2}}{\partial f_{i}\partial f_{j}}\left(  d_{ij}^{2}-d_{ji}
			^{2}\right)  =0
	\]
	since for any triangle with vertices $v_{i},v_{j},v_{k}$ we have
	\[
		d_{ij}^{2}-d_{ji}^{2}=d_{ik}^{2}+d_{kj}^{2}-d_{jk}^{2}-d_{ki}^{2}.
	\]
	We can compute that
	\[
		\left(  \frac{\partial}{\partial f_{i}}+\frac{\partial}{\partial f_{j}}\right)  d_{ij}
		=\frac{\partial d_{ij}}{\partial f_{i}}+\frac{\partial d_{ij}}{\partial f_{j}}
		=\frac{\partial d_{ij}}{\partial f_{i}}+\frac{\partial d_{ji}}{\partial f_{i}}=d_{ij}
	\]
	since
	\[
		\frac{\partial d_{ij}}{\partial f_{j}}=\frac{\partial^{2}\ell_{ij}}{\partial f_{i}f_{j}}
		=\frac{\partial d_{ji}}{\partial f_{i}}.
	\]
	It follows that
	\[
		\left(  \frac{\partial}{\partial f_{i}}+\frac{\partial}{\partial f_{j}
			}\right)  \left(  d_{ij}^{2}-d_{ji}^{2}\right)  
		=2\left(  d_{ij}^{2}-d_{ji}^{2}\right)
	\]
	and so it follows that
	\[
		\frac{\partial^{2}}{\partial^{2}f_{i}}\left(  d_{ij}^{2}-d_{ji}^{2}\right)
		=2\frac{\partial}{\partial f_{i}}\left(  d_{ij}^{2}-d_{ji}^{2}\right)
	\]
	and
	\[
		\frac{\partial^{2}}{\partial^{2}f_{j}}\left(  d_{ij}^{2}-d_{ji}^{2}\right)
		=2\frac{\partial}{\partial f_{j}}\left(  d_{ij}^{2}-d_{ji}^{2}\right)  .
	\]
	We can solve these equations, getting
	\begin{align*}
		\frac{\partial}{\partial f_{i}}\left(  d_{ij}^{2}-d_{ji}^{2}\right)    &
			=2a_{ij}e^{2f_{i}},\\
		\frac{\partial}{\partial f_{j}}\left(  d_{ij}^{2}-d_{ji}^{2}\right)    &
			=-2a_{ji}e^{2f_{j}}%
	\end{align*}
	for constants $a_{ij}$ and $a_{ji}.$ Hence
	\[
		d_{ij}^{2}-d_{ji}^{2}=a_{ij}e^{2f_{i}}-a_{ji}e^{2f_{j}}.
	\]
	We can now use (\ref{l deriv}) to find that for a constant $\eta_{ij}$
	\begin{align}
		\ell_{ij}^{2}=a_{ij}e^{2f_{i}}+a_{ji}e^{2f_{j}}+2\eta_{ij}e^{f_{i}+f_{j}}. \label{eqn: l-eucl}
	\end{align}
	From this, we compute that
	\[
		d_{ij}=\frac{\partial\ell_{ij}}{\partial f_{i}}=\frac{a_{ij}e^{2f_{i}}%
		+\eta_{ij}e^{f_{i}+f_{j}}}{\ell_{ij}}.
	\]
	We note that in a triangle, since
	\[
		d_{ij}^{2}-d_{ji}^{2}+d_{ki}^{2}-d_{ik}^{2}=d_{kj}^{2}-d_{jk}^{2}%
	\]
	and the right side is independent of $f_{i},$ differentiating with respect to
	$f_{i}$ gives
	\[
		2\left(  a_{ij}-a_{ik}\right)  e^{2f_{i}}=0
	\]
	and hence $a_{ij}=a_{ik}$ and $a$ is independent of the edge, only depending
	on the vertex, hence we rename $\alpha_i=a_{ij}=a_{ik}.$
	
	To see that the $\alpha_i$ and $\eta_{ij}$ must be consistent across triangles,
	consider Equation \ref{eqn: l-eucl} on both triangles and differentiate with respect to $f_i$ and
	$f_j$ to see that the $\eta_{ij}$ agree and then $f_i$ to see that the $\alpha_i$ agree.
\end{proof}

\section{Basic calculations in hyperbolic geometry}
\label{section:hyperbolic basics}

Before we move to the hyperbolic versions of the previous work, we
will review some techniques for computing in hyperbolic geometry.
This section summarizes the elementary facts about the hyperbolic plane 
$\H$ that we will
use in later calculations. All of the propositions in this section are
discussed in Chapter 3 of \cite{Ra}. See also \cite{Cho}. For the reader's convenience,
we have included some, but not all, proofs.

We use the hyperboloid model of $\H$ for the majority of our
calculations. In this model, the vector space $\R^3$ is equipped with a
Lorentzian inner product $\ast$ given by $u \ast v := u^TJv$ where $J$
is the diagonal matrix with entries 1,1,-1. We define a ``hyperbolic
magnitude'' $\|u\| := \sqrt{u \ast u}$; the only possible hyperbolic
lengths are nonnegative scalar multiples of 1 and $i$. $\H$
corresponds to those vectors $u = (u_1,u_2,u_3) \in \R^3$ satisfying
$u \ast u = -1$ and $u_3 > 0$.  
\begin{definition}
	A vector $u \in \R^3$ is termed
	\emph{spacelike} if $u \ast u > 0$, \emph{lightlike} (or ``on the light cone'') if
	$u \ast u = 0$, and \emph{timelike} if $u \ast u < 0$.
\end{definition}

The vector space structure on $(\R^3,\ast)$ gives us several ways to
describe a geodesic in $\H$:
\begin{itemize}
\item As a nonempty intersection $\H \cap \vecspan(p,q)$ for linearly
  independent $p,q \in \R^3$.
\item As a nonempty intersection $\H \cap p^\perp$, where $p$ is a
  spacelike vector and $p^\perp := \{ v \in \R^3 : p \ast v = 0 \}$.
\item As a path, parametrized by arclength, given by $\gamma(t) =
  \cosh(t) p + \sinh(t) v$. In this form, $p \in \H$, $v \in p^\perp$
  with $v \ast v = 1$. Note $p$ and $v$ encode the position and
  direction of $\gamma$ at $t = 0$.
\end{itemize}
The second characterization becomes particularly useful when combined
with the Lorentzian cross product, which is given by $p \otimes q :=
J(x \times y)$. Clearly, the Lorentzian cross product has two useful
properties:
\begin{itemize}
\item $p \otimes q = 0$ if and only if $p$ and $q$ are linearly dependent.
\item $p \otimes q$ is $\ast$-orthogonal to both $p$ and $q$.
\end{itemize}
A consequence of the second observation is that given distinct points
$p,q \in \H$, one simple way to describe the geodesic through $p$ and
$q$ is $(p \otimes q)^\perp$.

In the sequel, we will use $d_\H(u,v)$ to denote the hyperbolic distance
between two timelike points, and $d_\H(u,v^\perp)$ to denote the hyperbolic 
distance between a timelike point and a geodesic in hyperbolic space determined as 
the orthogonal complement of a spacelike point. 
When $u,v \in \R^3$ satisfy $|u \ast u|=|v \ast v| = 1$, we have the
following interpretations of the quantity $u \ast v$:
\begin{itemize}
\item If $u$ and $v$ are both timelike, then $u \ast v = - \cosh(d_\H(u,v))$.
\item If $u$ is timelike and $v$ is spacelike, then $u \ast v = \pm
  \sinh( d_\H(u,v^\perp) )$ and the sign depends upon which of the
  halfspaces bounded by $v^\perp$ contains $u$.
\item If $u$ and $v$ are both spacelike and $u^\perp$ and $v^\perp$
  intersect in angle $\alpha$ within $\H$, $u \ast v = \cos(\alpha)$.
\end{itemize}
Notice that the last item implies that for spacelike $u$ and $v$,
$u^\perp$ and $v^\perp$ meet at a right angle if and only if $u \ast v
= 0$.

The following identities simplify calculations that involve Lorentzian
cross products. Suppose $x,y,z,w \in \R^3$:
\begin{align}\label{eqn:cross product ident1}
x \otimes y &= - y \otimes x, \\
(x \otimes y) \ast z &= \det( x, y, z ), \\
x \otimes (y \otimes z) &= (x \ast y) z - (z \ast x) y, \\
(x \otimes y) \ast (z \otimes w) &= 
\begin{vmatrix} 
  x \ast w & x \ast z \\ 
  y \ast w & y \ast z 
\end{vmatrix}.\label{eqn:det formula}
\end{align}

We have already seen that several different kinds of data can be used
to specify a geodesic on $\H$. This allows us to extend our
understanding of where geodesics intersect.

\begin{definition}\label{definition intersection}
  Given a geodesics $\gamma$ on $\H$, we will identify $\gamma$ with
  the unique 2-dimensional subspace $P_\gamma$ of $\R^3$ such that
  $P_\gamma \cap \H$ is the image of $\gamma$.

  Given geodesics $\gamma,\omega$ on $\H$, we define their
  \emph{intersection} to be their intersection as subspaces of $\R^3$,
  namely $P_\gamma \cap P_\omega$.
\end{definition}

Readers familiar with the Klein model of $\H$ (the central projection
of $\H$ onto the plane $z=1$) should note that this definition is
simply a linear-algebraic way of formulating the notion of
intersecting 1-hyperplanes in the Klein model.

Introducing a broader notion of intersection allows us to generalize
familiar equations (like the law of cosines) and express them in terms
of linear algebra. Understanding how to interpret the Lorentzian inner
product is key to relating these different formulas. Often, the linear
algebraic interpretation allows us to efficiently treat several
seemingly different cases at once.

Recall the definition of a triangle (see Section 3.5 in \cite{Ra}),
which allows some of the vertices to be timelike, lightlike, or spacelike. 
We will concentrate on triangles with at least two timelike vertices.

\begin{proposition}\label{proposition ratcliffe triangles}
  Suppose $x \in \H$ and $y,z \in \R^3$ are either timelike or
  spacelike. Then
  \begin{align*}
    (z \otimes x) \ast (x \otimes y) = -\|z \otimes x\| \cdot \|x \otimes y \| \cos(\alpha),
  \end{align*}
  where $\alpha$ is the angle at $x$ in the (clockwise oriented)
  triangle $\{x,y,z\}$.
\end{proposition}

\begin{proposition}[The Generalized Law of Cosines]\label{proposition
    generalized law of cosines}
  Suppose $x,y,z \in \R^3$, with $\|x\| = \|z\| = i$ and $\|y\| = 1$
  or $i$, are the vertices of a triangle in
  $\H$, with angle $\alpha$ at $x$. Then
  \begin{align*}
    z \ast y + (z \ast x)(x \ast y) = \| z \otimes x \| \| x \otimes y \| \cos(\alpha).
  \end{align*}
\end{proposition}

\begin{proof}
  Assume, without loss of generality, that $x,y,z$ label the vertices
  of the triangle in clockwise order. Equation \ref{eqn:det formula}
  implies
  \begin{align*}
    -(z \otimes x) \ast (x \otimes y) = (z \ast y) + (z \ast x)(x \ast y).
  \end{align*}
  Now apply Proposition \ref{proposition ratcliffe triangles} to
  obtain the desired equality.
\end{proof}

By setting $\alpha = \pi/2$, we obtain a generalized version of the
Pythagorean theorem:
\begin{corollary}[The Generalized Pythagorean Theorem]
  Suppose $x,y,z$ are the vertices of a right triangle, with the right
  angle at $x$. Then:
  \begin{align*}
    -(z \ast y) = (z \ast x)(x \ast y).
  \end{align*}
\end{corollary}

We will require formulas for performing trigonometry in a hyperbolic
right triangle where one of the vertices (not the one adjacent to the
right angle) may be spacelike or timelike. Suppose we have a right
triangle labeled like the one in Figure \ref{fig:right triangle}.

\begin{figure}[ht]  
  \centering
  \includegraphics[scale=0.5]{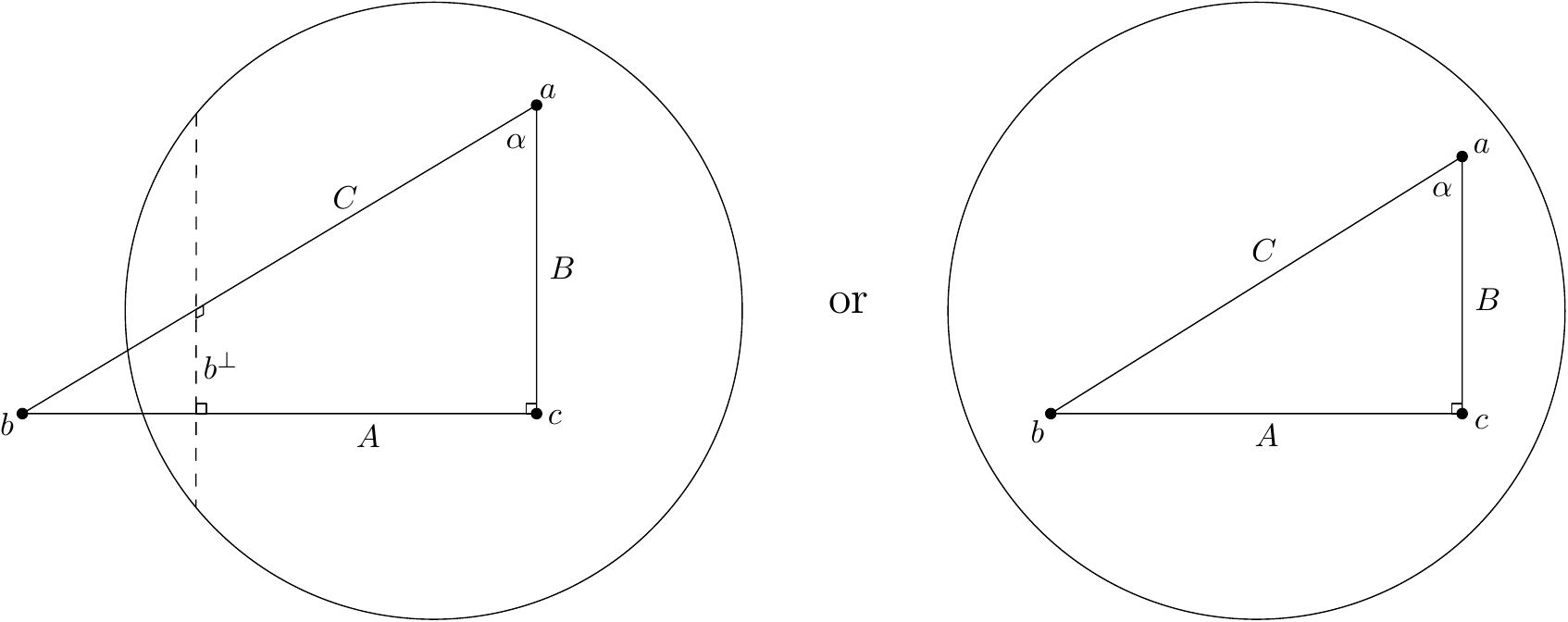}
  \caption{Two (Generalized) Right Triangles in the Klein Model}
  \label{fig:right triangle}
\end{figure}

\begin{proposition}\label{prop:trig}
  Given a triangle labeled as in Figure \ref{fig:right triangle}, we
  have:
  \begin{align*}
    \cos(\alpha) = \frac{\tanh(B)}{\tanh(C)},\hspace{0.25cm} 
    \sin(\alpha) = \frac{\sinh(A)}{\sinh(C)},\hspace{0.25cm} 
    \tan(\alpha) = \frac{\tanh(A)}{\sinh(B)}
  \end{align*}
  if $b$ is timelike and
  \begin{align*}
    \cos(\alpha) = \tanh(B)\tanh(C),\hspace{0.25cm} 
    \sin(\alpha) = \frac{\cosh(A)}{\cosh(C)},\hspace{0.25cm} 
    \tan(\alpha) = \frac{1}{\sinh(B)\tanh(A)}
  \end{align*}
  if $b$ is spacelike.
\end{proposition}

Deriving these formulas is an easy application of the generalized
Pythagorean theorem and the generalized law of cosines.

The next corollary generalizes the familiar formula for the
cosine of an angle in a hyperbolic right triangle.
\begin{corollary}\label{corollary cosine right triangle}
  Suppose $x,y,z$ are the vertices of a right triangle (with the right
  angle at $z$) satisfying the assumptions of Proposition
  \ref{proposition generalized law of cosines}. Then
  \begin{align*}
    \cos(\alpha) = -\frac{x \ast y}{\| x \otimes y\|} \tanh( d_\H(z,x) ).
  \end{align*}
\end{corollary}

\begin{proof}
  The right angle at $z$ means that:
  \begin{align*}
    0 &= (y \otimes z) \ast (z \otimes x) \\
    &= -(y \ast x) - (y \ast z)(z \ast x) 
   \end{align*}
   and so
   \[
    z \ast y = -\frac{y \ast x}{z \ast x}.
  \]

  Substituting this into the equation we obtain from the Law of
  Cosines, we learn:
  \begin{align*}
    \| z \otimes x \| \| x \otimes y \| \cos(\alpha) 
    &= z \ast y + (z \ast x)(x \ast y) \\
    &= -\frac{y \ast x}{z \ast x} + (z \ast x)(x \ast y) \\
    &= (x \ast y)\frac{(z \ast x)^2 - 1}{z \ast x} \\
    &= (x \ast y)\frac{\sinh^2(d_\H(z,x))}{-\cosh(d_\H(z,x))}.
  \end{align*}
  Using Equation \ref{eqn:det formula}, it is easy to
  check $\| z \otimes x \| = \sinh(d_\H(z,x))$. Hence:
  \begin{align*}
  \cos(\alpha) = -\frac{x \ast y}{\| x \otimes y\|} \tanh( d_\H(z,x) ).
  \end{align*}
\end{proof}

Because the Lorentzian inner product is nondegenerate, we have a well
defined notion of $\ast$-orthogonality and may apply the Gram-Schmidt
procedure to obtain a basis of mutually $\ast$-orthogonal vectors. This
procedure can be used to parametrize a geodesic given in the form $\H
\cap \vecspan(p,q)$ by arclength.

\begin{proposition}\label{proposition Gram-Schmidt}
  Suppose $p \in \H$, and $q \in \R^3$. Then the geodesic $\H \cap
  \vecspan(p,q)$ may be parametrized by arclength as:
  \begin{align*}
    \gamma(t) = \cosh(t) p + \sinh(t) \frac{q + (p \ast q)p}{\sqrt{q \ast q + (p \ast q)^2}}.
  \end{align*}
\end{proposition}

\begin{proof}
  The geodesic in question can be parametrized by arclength as
  $\gamma(t) = \cosh(t)p + \sinh(t)v$ for some spacelike $v$ with $v
  \ast v = 1$; we simply need to use the Gram-Schmidt procedure to
  guarantee that $\vecspan(p,v) = \vecspan(p,q)$ and $v \in p^\perp$.

  So consider the vector $q + (p \ast q) p$. Notice $-(p \ast q)p$ is
  the $\ast$-projection of $q$ onto the subspace spanned by $p$, and
  \begin{align*}
    p \ast (q + (p \ast q) p) = p \ast q - p \ast q = 0.
  \end{align*}
  To find $v$, we only need to rescale this projection. Since
  \begin{align*}
    (q + (p \ast q) p) \ast (q + (p \ast q) p) 
    &= q \ast q + 2 (p \ast q)^2 + (p\ast q)^2 (p \ast p) \\
    &= q \ast q + (p \ast q)^2
  \end{align*}
  the appropriate $v$ is
  \begin{align*}
    v = \frac{q + (p \ast q) p}{\sqrt{q \ast q + (p \ast q)^2}}.
  \end{align*}
\end{proof}

\section{Duality structures on hyperbolic triangles}
\label{section:duality hyp}
%

We interpret a piecewise hyperbolic pre-metric as subdividing each
edge $\{i,j\}$ of length $\ell_{ij}$ into two portions of length $\dij$ and $\dji$, that are
assigned to the vertices $i$ and $j$ respectively. 

\begin{definition}
  \label{def:edge centers}
  Given a pre-metric $d$ and an isometric embedding of a simplex
  $\{i,j\}$ into $\H$:
  \begin{itemize}
  \item The \emph{vertices} $p_i,p_j \in \H$ of $\{i,j\}$ are the
    images of $i$ and $j$ under the embedding.
  \item The \emph{edge center $c_{ij}$ induced by $d$} is the unique
    point along the line $E_{ij}$ through $p_i$ and $p_j$ such that
    $c_{ij}$ is (signed) distance $d_{ij}$ from $p_i$ and $d_{ji}$ from $p_j$.
  \item The \emph{edge perpendicular $P_{ij}$} is the line through
    $c_{ij}$ that is orthogonal to $E_{ij}$.
  \end{itemize}
\end{definition}

Unlike in the Euclidean setting, it is possible that the geodesics
$P_{ij}$ and $P_{jk}$ do not intersect within $\H$. However, these two
1-hyperplanes can be understood as intersecting in the more general
sense of Definition \ref{definition intersection}, namely the
two-dimensional subspaces of $(\R^3,\ast)$ associated to $P_{ij}$ and
$P_{jk}$ intersect in a one-dimensional subspace. One can then ask for
necessary and sufficient conditions on the pre-metric that guarantee
that for each simplex $\{i,j,k\}$
\begin{align}\label{eqn:duality intersections}
  P_{ij} \cap P_{jk} = P_{jk} \cap P_{ki} = P_{ki} \cap P_{ij}
\end{align}
or, colloquially, the three perpendiculars of $\{i,j,k\}$ intersect in
a single point (this point is in the span of $\{i,j,k\}$). 
This condition can also be interpreted in the Klein
model of hyperbolic space as the condition that the three lines 
representing the geodesics intersect at the same point in the plane
of the Klein model.

\begin{proposition}
  \label{prop:hyperbolic compatibility}
  Suppose $d$ is a piecewise hyperbolic pre-metric. Equation
  \ref{eqn:duality intersections} holds if and only if the following
  \emph{compatibility equation}
  \begin{align}\label{eqn:dijhyp-reln}
    (p_i \ast c_{ij})(p_j \ast c_{jk})(p_k \ast c_{ki}) 
    = (p_i \ast c_{ki})(p_j \ast c_{ij})(p_k \ast c_{jk})
  \end{align}
  is satisfied for every simplex $\{i,j,k\}$.
  
  Since the vectors $p_i$ and $c_{ij}$ are timelike of length -1,
  Equation \ref{eqn:dijhyp-reln} has the following equivalent
  formulation:
    \begin{align*}
      \cosh( d_{ij} ) \cosh( d_{jk} ) \cosh( d_{ki} ) = 
      \cosh( d_{ji} ) \cosh( d_{kj} ) \cosh( d_{ik} ).
    \end{align*}
\end{proposition}

\begin{proof}
  To simplify our notation, we shall consider a single 2-simplex
  $\{1,2,3\}$. The vertices of the embedded 2-simplex are linearly
  independent vectors $p_1, p_2, p_3 \in \H$. 

  Consider that if $c$ is a point on the perpendicular $P_{ij}$, then
  $P_{ij} = (c \otimes c_{ij})^\perp$. Likewise the span of edge $e_{ij}$ is given by $(p_i \otimes
  p_j)^\perp$. Since $c_{ij}$ belongs to both $P_{ij}$ and $e_{ij}$,
  the fact that $P_{ij}$ and $e_{ij}$ are perpendicular is equivalent
  to the equation:
  \begin{align*}
    (c \otimes c_{ij}) \ast ( p_i \otimes p_j ) = 0.
  \end{align*}
  Identities \ref{eqn:cross product ident1}-\ref{eqn:det formula}
  imply this is equivalent to the equation:
  \begin{align*}
    c \ast ( (c_{ij} \ast p_i)p_j - (c_{ij} \ast p_j)p_i) = 0.
  \end{align*}  
  Hence, Equation \ref{eqn:duality intersections} holds for simplex
  $\{1,2,3\}$ if and only if there is a nontrivial solution $c$ to the
  system:
  \begin{align*}
    c \ast ( (c_{12} \ast p_1)p_2 - (c_{12} \ast p_2)p_1 ) &= 0 \\
    c \ast ( (c_{23} \ast p_2)p_3 - (c_{23} \ast p_3)p_2 ) &= 0 \\
    c \ast ( (c_{31} \ast p_3)p_1 - (c_{31} \ast p_1)p_3 ) &= 0
  \end{align*}
  This system can be reformulated as a matrix equation
  \begin{align*}
    \begin{bmatrix}
      ((c_{12} \ast p_1)p_2 - (c_{12} \ast p_2)p_1)^T\\
      ((c_{23} \ast p_2)p_3 - (c_{23} \ast p_3)p_2)^T\\ 
      ((c_{31} \ast p_3)p_1 - (c_{31} \ast p_1)p_3)^T     
    \end{bmatrix} \cdot J \cdot c = 0
  \end{align*}
  that has a nontrivial solution if and only if the determinant of the
  first matrix is zero. Expanding that determinant and canceling the
  (nonzero) factors of $\det(p_1,p_2,p_3)$ that arise yields Equation
  \ref{eqn:dijhyp-reln}. The last statement follows easily.
\end{proof}

This proposition motivates the hyperbolic case of Definition \ref{def: metric}.

\begin{remark} 
  One can gain insight into how the Euclidean and
  hyperbolic compatibility conditions are related by comparing 
  Equation \ref{d-euclidean condition} and Equation \ref{d-hyperbolic condition}
  for small $d_{ij}$ in the same way one compares the Euclidean
  Pythagorean Theorem with the hyperbolic version, $\cosh(c) =
  \cosh(a)\cosh(b)$.
\end{remark}

%
%

\section{Conformal variations of hyperbolic triangles}
\label{section:conformal hyp}

Various formulations of conformal variations of hyperbolic
triangulations of surfaces have been studied in \cite{Thurs},
\cite{MR}, \cite{CdV}, \cite{CL}, \cite{Spr}, \cite{Guo},
\cite{BPS}, \cite{ZGZLYG}.  We present a unified approach from the
perspective of the metric triangulations as defined above. 
\subsection{Motivation and variation formula}

Suppose we wanted to generate a metric from weights assigned to vertices, so that
$d_{ij}=d_{ij}(f_i,f_j)$ for some function $f$ on the vertices. If this our starting 
point for conformal structure, in order to compute conformal variations, we will consider what
happens to the metric on a triangle $\{1,2,3\}$
 when the conformal parameter $f_3$ changes but the other two
do not, i.e., $\delta f_1=\delta f_2=0$. We will call this a
$f_3$-conformal variation in this section.

The next two propositions analyze the configuration shown in Figure
\ref{figure conformal variation}. We assume throughout that
$v_1,v_2,v_3$ are linearly independent in $\R^3$, with $v_i \ast v_i =
-1$.

\begin{figure}
  \centering
  \includegraphics[scale=0.75]{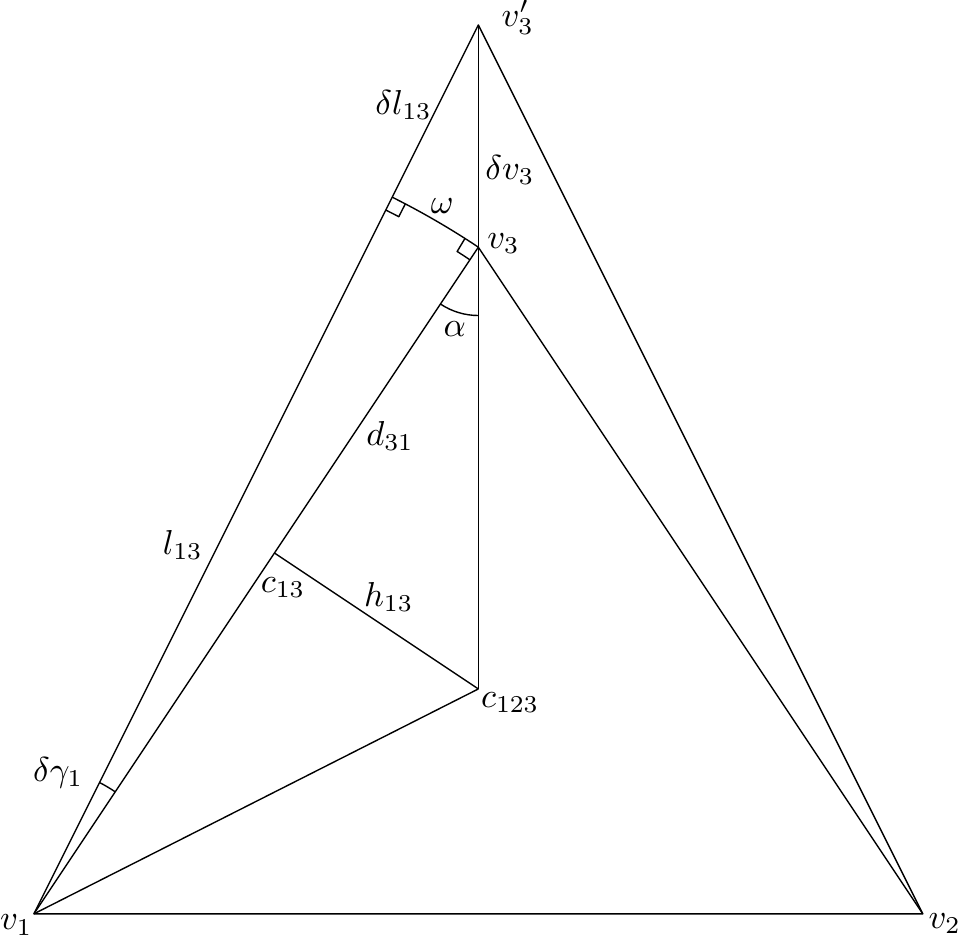}
  \caption{Variation of a Hyperbolic Triangle}
  \label{figure conformal variation}
\end{figure}

\begin{proposition}\label{prop:inner product identities}
  Under an $f_3$-conformal variation:
  \begin{align*}
    v_1 \ast \delta v_3 &= -\sinh \ell_{13} \frac{\partial \ell_{13}}{\partial f_3} \delta f_3\\
    v_2 \ast \delta v_3 &= -\sinh \ell_{23} \frac{\partial \ell_{23}}{\partial f_3} \delta f_3\\
    v_3 \ast \delta v_3 &= 0
  \end{align*}
\end{proposition}

\begin{proof}
  Bilinearity of the Lorentzian inner product implies:
  \begin{align*}
    \delta ( v_1 \ast v_3 ) = v_1 \ast \delta v_3.
  \end{align*}
  However, since $v_1 \ast v_3 = - \cosh( \ell_{13} )$, we can also write:
  \begin{align*}
    \delta ( v_1 \ast v_3 ) = 
    -\sinh\ell_{13} \frac{\partial\ell_{13}}{\partial f_3}  \delta f_3.
  \end{align*}
  Hence
  \begin{align*}
    v_1 \ast \delta v_3 = - \sinh \ell_{13}  \frac{\partial\ell_{13}}{\partial f_3} \delta f_3.
  \end{align*}
  We get the formula for $v_2 \ast \delta v_3$ similarly.
  Finally, since $v_3 \ast v_3 = -1$:
  \begin{align*}
    0 = \delta( v_3 \ast v_3 ) = 2 v_3 \ast \delta v_3.
  \end{align*}
\end{proof}

The following proposition makes precise what we mean by the colloquial statement
that conformal variations give good angle variations.
\begin{proposition}\label{prop:conformal line}
  Let $c_{123}$ denote the center of the triangle specified by the
  vertices $v_i$ and the (compatible) partial edge lengths
  $d_{ij}$. Suppose further that the edge centers on edges $\{1,3\}$
  and $\{2,3\}$ are timelike. Then under a $f_3$-conformal variation,
  the points $v_3', v_3$ and $c_{123}$ lie on a line in $\H$ if and
  only if $\frac{\partial \ell_{ij}}{\partial f_i} = (\tanh d_{ij})
  F(f_i)$, for some function $F(f_i)$.
\end{proposition}

\begin{proof}
  Without loss of generality, assume $c_{123} \ast c_{123} = \pm
  1$. Consider the geodesic through $v_3$ and $c_{123}$. As a set,
  this geodesic can be described by $\H \cap \vecspan(v_3,c_{123})$, a
  characterization we will use to parametrize the geodesic by
  arclength as $\cosh(t) v_3 + \sinh(t) u$ for some $u \in v_3^\perp
  \cong T_{v_3}\H$. Specifically, Proposition \ref{proposition
    Gram-Schmidt} implies:
  \begin{align*}
    u = \frac{c_{123} + (v_3 \ast c_{123}) v_3} {\sqrt{c_{123}\ast
        c_{123}+(v_3 \ast c_{123})^2}}
  \end{align*}

  The points $v_3', v_3$, and $c_{123}$ lie on a geodesic if and only
  if $u$ and $\delta v_3$ are collinear. The three numbers $\{ \delta
  v_3 \ast v_i \}_{i=1}^3$ completely characterize the vector $\delta
  v_3 \in v_3^\perp \cong T_{v_3}\H$. Hence, $u$ and $\delta v_3$ are
  collinear if and only if there exists $\lambda \in \R$ such that $u
  \ast v_i = \lambda\, \delta v_3 \ast v_i$ for $i=1,2,3$. We already
  know $v_3 \ast u = v_3 \ast \delta v_3 = 0$, so only $v_1 \ast u$
  and $v_2 \ast u$ require consideration.

  Consider our equation for $u$. The scalar in the denominator will
  appear in both $v_1 \ast u$ and $v_2 \ast u$. To simplify our
  notation, we will write $\lambda_3 := (c_{123} \ast c_{123} + (v_3
  \ast c_{123})^2)^{-1/2}$. Now
  \begin{align*}
    v_1 \ast u = \lambda_3 ( c_{123} \ast v_1 + (v_3 \ast c_{123})(v_3 \ast v_1) )
  \end{align*}
  and we can apply the Generalized Law of Cosines (Proposition
  \ref{proposition generalized law of cosines}) to the triangle
  $\{v_1,v_3,c_{123}\}$ in order to rewrite this equation as
  \begin{align*}
    v_1 \ast u 
    &= \lambda_3 \|v_1 \otimes v_3\|\|v_3 \otimes c_{123}\|\cos(\alpha) \\
    &= \lambda_3 \sinh(\ell_{13}) \|v_3 \otimes c_{123}\|\cos(\alpha).
  \end{align*}
  Next consider the right triangle with vertices $\{ c_{123}, c_{13},
  v_3 \}$. By Corollary \ref{corollary cosine right triangle}, we
  have
  \begin{align*}
    \| v_3 \otimes c_{123} \| \cos(\alpha) = -(v_3 \ast c_{123})\tanh( d_{31} ).
  \end{align*}
  A final substitution into our equation for $v_1 \ast u$ implies
  \begin{align*}
    v_1 \ast u = -(\lambda_3 \cdot v_3 \ast c_{123}) \sinh(\ell_{13})\tanh(d_{31}).
  \end{align*}
  A similar argument for $v_2$ yields
  \begin{align*}
    v_2 \ast u = -(\lambda_3 \cdot v_3 \ast c_{123}) \sinh(\ell_{23})\tanh(d_{32}).
  \end{align*}
  From Proposition \ref{prop:inner product identities},
  we know that for $k=1,2$
  \begin{align*}
    v_k \ast \delta v_3 = -\sinh \ell_{k3} \frac{\partial
      \ell_{k3}}{\partial f_3} \delta f_3.
  \end{align*}
  Comparing these two equations, we see that there exists $\lambda \in
  \R$ so that $v_k \ast u = \lambda v_k \ast \delta v_3$ if and only
  if there exists a smooth function $F(f_3)$ for which
  $\frac{\partial \ell_{k3}}{\partial f_3} = \tanh(d_{3k})F(f_3)$.
\end{proof}

Proposition \ref{prop:conformal line} motivates the hyperbolic case of Definition 
\ref{def:conformal structure}, where we have chosen to simplify to parameters
that make $F$ equal to the constant function $1$.

We will now study how the angles change under a conformal
variation. First we see the following.

\begin{theorem} \label{thm: hyperbolic angle variation}
  Given a conformal structure, then for any simplex $\{i,j,k\}$
  \begin{align}
    \frac{\partial \gamma_i}{\partial f_j} 
    &= \frac{1}{\cosh d_{ji}}\frac{\tanh^\beta h_{ij}}{\sinh \ell_{ij}} 
    \label{eqn:angle variation ij thmcopy} \\ 
    \frac{\partial \gamma_i}{\partial f_i} 
    &= -\frac{\partial A_{ijk}}{\partial f_i}
    -\frac{\partial \gamma_j}{\partial f_i}
    -\frac{\partial \gamma_k}{\partial f_i} 
    \label{eqn:angle variation ii thmcopy}
    \end{align}
  where $\beta$ is 1 if $c_{ijk}$ is timelike and -1 if $c_{ijk}$ is
  spacelike.
\end{theorem}

\begin{proof}
  For simplicity, we shall consider the problem for a single simplex
  $\{1,2,3\}$ labeled as in Figure \ref{figure conformal variation},
  with $i = 1$, $j = 3$. We will address the case where $c_{123}$ is
  timelike; the case where $c_{123}$ is spacelike is similar. Once
  \ref{eqn:angle variation ij thmcopy} is proven, \ref{eqn:angle
    variation ii thmcopy} follows immediately because of the area
  formula for a hyperbolic triangle:
  \begin{align*}
    A_{123} = \pi - \gamma_1 - \gamma_2 - \gamma_3.
  \end{align*}
  
  Because the variation is conformal, $\delta \ell_{13} =
  \tanh d_{31} \delta f_3$. Using the formula for a segment of a
  circle in the hyperbolic plane, we have $\omega = \delta \gamma_1
  \sinh \ell_{13}$.

  By Proposition \ref{prop:conformal line}, under a conformal
  variation $v_3,v_3'$ and $c_{123}$ are collinear. Consequently, the
  angle adjacent to $v_3$ in the triangle with side lengths $\omega$,
  $\delta l_{13}$ and $\delta v_3$ is $\pi/2 - \alpha$. This, together
  with the formulas in Proposition \ref{prop:trig}, allows us to write:
  \begin{align*}
    \tan(\alpha) &= \frac{\tanh h_{13}}{\sinh d_{31}},\\
    \cot(\alpha) &= \tan\left( \frac{\pi}{2} - \alpha \right) = \frac{\tanh\delta \ell_{13}}{\sinh \omega},\\
  \frac{\tanh h_{13}}{\sinh d_{31}} &= \frac{\sinh \omega}{\tanh\delta \ell_{13}}
  = \frac{\sinh(\delta \gamma_1 \sinh\ell_{13})}{\tanh(\delta f_3 \tanh d_{31})}. 
  \end{align*}

  Using the Taylor series for $\sinh$ and $\tanh$, we have:
  \begin{align*}
    \frac{\tanh h_{13}}{\sinh d_{31}} 
    &= \frac{ \delta \gamma_1 \sinh\ell_{13} 
      + \bigO(\delta \gamma_1^3)}{\delta f_3 \tanh d_{31}+ \bigO(\delta f_3^3)}, 
  \end{align*}   
  and hence,
  \begin{align*}
    \frac{\delta \gamma_1}{\delta f_3} 
    &= \frac{1}{\cosh d_{31}}
       \frac{\tanh h_{13}}{\sinh \ell_{13}}
       \left(\frac{1 + \bigO(\delta f_3^2)}{1 + \bigO(\delta \gamma_1^2)}\right).
  \end{align*}
\end{proof}
We can also compute the variation of area explicitly. 
\begin{proposition} \label{prop:hyperbolic area deriv}
  Given a conformal structure, then for any simplex $\{i,j,k\}$ with area
  $A_{ijk}$.
  \begin{align}
    \frac{\partial A_{ijk}}{\partial f_k} 
    &= \frac{\partial \gamma_i}{\partial f_k} (\cosh \ell_{ik}-1)
    	+2\frac{\partial \gamma_j}{\partial f_k}(\cosh \ell_{jk}-1).
    \end{align}
    In particular, if the derivatives $\partial \gamma_i/\partial f_k$ are positive whenever
     $k \neq i$, then the derivative of the area is positive.
\end{proposition}
\begin{proof}
	This follows from the formula for the area of a sector of circle as a function of the 
	radius for a hyperbolic surface, since in Figure \ref{figure conformal variation} we
	find that the area of each of the small triangles is higher order, leaving only 
	the areas of the skinny triangles in the picture.
\end{proof} 
\subsection{Characterization of discrete conformal structures}\label{section: hyp conf classify}
The proof of the hyperbolic case of Theorem \ref{thm:conformal classification}
is similar to the proof of the Euclidean case, 
though the calculation is a bit harder in hyperbolic background.
\begin{proof}[Proof of the hyperbolic case of Theorem \ref{thm:conformal classification}]
	We first note the following:
	\begin{align}
		\frac{\partial}{\partial f_{i}}\cosh\ell_{ij}
		& =\cosh\ell_{ij}-\frac{\cosh d_{ji}}{\cosh d_{ij}}, \label{eq:coshder}\\
		\frac{\partial}{\partial f_{j}}\cosh\ell_{ij}
				& =\cosh\ell_{ij}-\frac{\cosh d_{ij}}{\cosh d_{ji}}. \label{eq:coshder2}
	\end{align}
		
	A straightforward calculations gives that
	\[
		\left(  \frac{\cosh^{2}d_{ij}}{\cosh^{2}d_{ji}}\frac{\partial}{\partial f_{i}}
			+\frac{\partial}{\partial f_{j}}\right)  
				\log\frac{\cosh^{2}d_{ij}}{\cosh^{2}d_{ji}}
		=2\left(  \frac{\cosh^{2}d_{ij}}{\cosh^{2}d_{ji}}-1\right)
	\]
	or if $H=\log\frac{\cosh^{2}d_{ij}}{\cosh^{2}d_{ji}}$ then
	\[
		\left( e^{H}\frac{\partial}{\partial f_{i}}+\frac{\partial}{\partial f_{j}}\right) H
		=2\left(  e^{H}-1\right).
	\]
	Since
	\[
		\frac{\partial^{2}H}{\partial f_{i}\partial f_{j}}=0
	\]
	it follows that
	\[
		e^{H}\frac{\partial^{2}H}{\partial f_{i}^{2}}+e^{H}\left(  \frac{\partial
			H}{\partial f_{i}}\right)  ^{2}=2e^{H}\frac{\partial H}{\partial f_{i}}%
	\]
	and
	\[
		e^{-H}\frac{\partial^{2}H}{\partial f_{j}^{2}}-e^{-H}\left(  \frac{\partial
			H}{\partial f_{j}}\right)  ^{2}=2e^{-H}\frac{\partial H}{\partial f_{j}}.
	\]

	One can then easily solve this ODE to obtain that:
	\[
	  \frac{\partial H}{\partial f_{i}} = 2\frac{a_{ij} e^{2f_i}}{1 + a_{ij} e^{2f_i}}
	\]
	for some constant $a_{ij}$ and
	\[
	  \frac{\partial H}{\partial f_{j}} = -2\frac{a_{ji} e^{2f_j}}{1 + a_{ji} e^{2f_j}}
	\]
	for some constant $a_{ji}$. It follows that 
	\begin{align}
		\frac{\cosh^2 d_{ij}}{\cosh^2 d_{ji}}= D 
			\frac{1+a_{ij} e^{2f_{i}}}{1+a_{ji} e^{2f_{j}}} 
			\label{eqn:cosh dij by cosh dij}.
	\end{align}
	
	We can now use Equation \ref{eq:coshder} to see that 
	\begin{align*}
	  	\cosh\ell_{ij}-\frac{\partial}{\partial
			f_{i}}\cosh\ell_{ij} &=\frac{1}{D}\left(  \frac{1+a_{ij} e^{2f_{i}}}{1
			+a_{ji} e^{2f_{j}}}\right)  ^{-1/2}\\
		\cosh\ell_{ij}-\frac{\partial}{\partial
					f_{j}}\cosh\ell_{ij} &=D\left(  \frac{1+a_{ij} e^{2f_{i}}}{1
					+a_{ji} e^{2f_{j}}}\right)  ^{1/2}\\
	\end{align*}
	and so we find that $D=1$ and
	\begin{align}
		\cosh\ell_{ij}=\sqrt{\left(  1+a_{ji}e^{2f_{j}}\right)  \left(  1+a_{ij}%
		e^{2f_{i}}\right)  }+\eta_{ij}e^{f_{i}+f_{j}} \label{eqn: hyp length}
	\end{align}
	for some constant $\eta_{ij}$.
	
	The compatibility condition (\ref{d-hyperbolic condition}) implies that 
	$\log \frac{\cosh d_{ij}}{\cosh d_{ji}} + \log \frac{\cosh d_{ki}}{\cosh d_{ik}}$
	is independent of $f_i$ and so we can use Equation \ref{eqn:cosh dij by cosh dij}
	to see that $a_{ij}=a_{ik}$ and so we can define $\alpha_i=a_{ij}=a_{ik}$. 
		
	It follows that 
	  \begin{align*}
	    \tanh d_{ij} &= \frac{1}{\sinh \ell_{ij}} \frac{\partial}{\partial f_i} \cosh \ell_{ij}\\  
	      & = \frac{\alpha_{i}e^{2f_{i}}}{\sinh
					\ell_{ij}}\sqrt{\frac{1+\alpha_{j}e^{2f_{j}}}{1+\alpha_{i}e^{2f_{i}}}}
					+\frac{\eta_{ij}e^{f_{i}+f_{j}}}{\sinh\ell_{ij}}.
	  \end{align*}
	
	Finally, we can use Equation \ref{eqn:cosh dij by cosh dij} again to write 
	$2\log \frac{\cosh d_{ij}}{\cosh d_{ji}}$ in terms of the coefficients determined
	in the two triangles adjacent to edge $\{i,j\}$ and differentiate to see that the
	$\alpha_i$ derived in each triangle must be equal. It then follows from Equation 
	\ref{eqn: hyp length} that the $\eta_{ij}$ derived in each triangle must be equal
	as well.

\end{proof}
\subsection{Variational formulation for curvature}
\label{subsect:variational hyp}
While the formula (\ref{eqn:angle variation ij}) is not symmetric in $i$ and $j$, we can 
reparametrize to get a symmetric variation formula. Notice that Equation \ref{eqn:cosh dij by cosh dij} 
(recall that we proved $D=1$) implies that 
\[
	\frac{\sqrt{1+\alpha_i e^{2f_i}}}{\cosh d_{ij}}=\frac{\sqrt{1+\alpha_j e^{2f_j}}}{\cosh d_{ji}}.
\]
If we take new coordinates $u_i = u_i(f_i)$ such that
\[
	\frac{\partial f_i}{\partial u_i} = \sqrt{1+\alpha_i e^{2f_i}}
\]
then we have the symmetry
\[
	\frac{\partial \gamma_i}{\partial u_j}=\frac{\partial \gamma_j}{\partial u_i}.
\]
\begin{remark}
	The function $u_i(f_i)$ can be computed explicitly. It is not hard to see that if $\alpha_i=0$ then 
	$u_i=f_i$ and if not then 
	\[
		u_i = \frac12 \log \left| \frac{\sqrt{1+\alpha_ie^{2f_i}}-1}{\sqrt{1+\alpha_ie^{2f_i}}+1}\right|.
	\]
	If $\alpha_i <0$ then this is
	\[
		-\tanh u_i = \sqrt{1+\alpha_ie^{2f_i}}
	\]
	and if $\alpha_i >0$ then this is
	\[
		-\coth u_i = \sqrt{1+\alpha_ie^{2f_i}}.
	\]
	Compare to the formulations in \cite{Guo}, \cite{BPS}, and \cite{ZGZLYG}.
\end{remark}

It then follows that for a triangle $t=\{1,2,3\}$ the following form
is closed:
\begin{align}
  \omega_{t} = \sum_{i=1}^3 \gamma_i d u_i.  \label{eqn:omega_T}
\end{align}
We can now integrate to get a function on the whole triangulation, where we fix some $\bar{u}$:
\begin{align}
  F(u)=2\pi \sum_i u_i - \sum_t \int_{\bar{u}}^u \omega_t. \label{eqn:F}
\end{align}

\begin{theorem}\label{thm:hyp functional}
  The function $F$ has the property that 
  \[
  \frac{\partial F}{\partial u_i} = K_i.
  \]
  Furthermore, if all $d_{ij}>0$ and $h_{ij}>0$ then this function is strictly convex.
\end{theorem}
\begin{proof}
	The first statement follows from the definition. The second follows from the facts that in a
	triangle $\{1,2,3\}$, 
	\begin{align*} 
		\frac{\partial \gamma_i}{\partial u_j} &\geq 0 \\
		\left| \frac{\partial \gamma_i}{\partial u_i} \right| &>
			\frac{\partial \gamma_i}{\partial u_j}+\frac{\partial \gamma_i}{\partial u_k} 
	\end{align*}
	for $\{i,j,k\}=\{1,2,3\}$ since
	\[
		 \frac{\partial A_{123}}{\partial f_i}>0
	\]
	by Proposition \ref{prop:hyperbolic area deriv}. It follows that the matrix of partial derivatives
	is diagonally dominant.
\end{proof}

\section{Spherical Geometry} \label{sec: sphere}

The arguments presented in the case of hyperbolic background geometry
can be adjusted for the case of spherical background
geometry. Essentially, this occurs because in the hyperbolic case we
are studying properties of the Lorentzian inner product $\ast$, while
in spherical geometry we study analogous properties of the Euclidean
inner product. Because the definitions and arguments in the spherical
case are so similar to those of previous sections, we will only state
the main results in the spherical case.

To work in the spherical case, we work with the usual dot product $\cdot$
on $\R^3$. Geodesics on the sphere correspond to planes in $\R^3$ and so 
given a triangle $\{i,j,k\}$ in the sphere and a pre-metric, a given embedding 
induces planes $P_{ij}$, etc. through edge centers and the condition for
inducing a duality structure is 
\begin{align}\label{eqn:spherical duality intersections}
  P_{ij} \cap P_{jk} = P_{jk} \cap P_{ki} = P_{ki} \cap P_{ij}
\end{align}

As in the hyperbolic case, this corresponds to a compatibility condition on the
partial edge lengths.
\begin{proposition}
  \label{prop:spherical compatibility}
  Suppose $d$ is a piecewise spherical pre-metric. Equation
  \ref{eqn:spherical duality intersections} holds if and only if the
  following \emph{compatibility equation}
  \begin{align}\label{eqn:dij-reln spherical}
    (p_i \cdot c_{ij})(p_j \cdot c_{jk})(p_k \cdot c_{ki}) 
    = (p_i \cdot c_{ki})(p_j \cdot c_{ij})(p_k \cdot c_{jk})
  \end{align}
  is satisfied for every simplex $\{i,j,k\}$. Equation
  \ref{eqn:dij-reln spherical} has the following equivalent formulation:
    \begin{align*}
      \cos( d_{ij} ) \cos( d_{jk} ) \cos( d_{ki} ) = 
      \cos( d_{ji} ) \cos( d_{kj} ) \cos( d_{ik} ).
    \end{align*}
\end{proposition}


We can also look at discrete conformal structures. The angle variation
theorem takes the following form.

\begin{theorem} \label{thm: spherical angle variation}
  Given a conformal structure, then for any simplex $\{i,j,k\}$:
  \begin{align}
    \frac{\partial \gamma_i}{\partial f_j} 
    &= \frac{1}{\cos d_{ji}}\frac{\tan h_{ij}}{\sin \ell_{ij}} \\ 
    \frac{\partial \gamma_i}{\partial f_i} 
    &= \frac{\partial A_{ijk}}{\partial f_i}
    -\frac{\partial \gamma_j}{\partial f_i}
    -\frac{\partial \gamma_k}{\partial f_i}.    
  \end{align}
\end{theorem}
Note that although the heights $h_{ij}$ require choosing one of the two possible centers, 
the term $\tan h_{ij}$ does not depend on this choice, since choosing the other center
leads to heights $h'_{ij}=-(\pi - h_{ij})$ and so $\tan h'_{ij}=\tan h_{ij}$

We can also compute the variation of area explicitly.
\begin{proposition}\label{prop: sphere area var}
  Given a spherical conformal structure, then for any simplex
  $\{i,j,k\}$ with area $A_{ijk}$, we have
  \begin{align*}
   \frac{\partial A_{ijk}}{\partial f_k} 
   = \frac{\partial \gamma_i}{\partial f_k}(1- \cos \ell_{ik})
     + \frac{\partial \gamma_j}{\partial f_k} (1-\cos \ell_{jk}).
   \end{align*}
\end{proposition}

Using this theorem and the definition of a spherical conformal
structure, one can derive the spherical case of Theorem
\ref{thm:conformal classification}.  As in the hyperbolic case, it is
desirable to change from the variables $f_i$ to variables $u_i =
u_i(f_i)$, so that one can recognize that $\partial \gamma_i
/ \partial u_j = \partial \gamma_j / \partial u_i$. The variables
$u_i$ are given by
\begin{align*}
  \frac{\partial f_i}{\partial u_i} = \sqrt{1 - \alpha_i e^{2f_i}}
\end{align*}  

Finally, we may define closed forms $\omega_t$ and a function $F$ as
in Equations \ref{eqn:omega_T} and \ref{eqn:F}. We have the following
analog to Theorem \ref{thm:hyp functional}.

\begin{theorem} \label{thm:sphere functional}
  The function $F$ has the property that 
  \begin{align*}
    \frac{\partial F}{\partial u_i} = K_i.
  \end{align*}
\end{theorem}
Notice that we do not have a corresponding notion of convexity for
this functional, as we do in the cases of Euclidean and hyperbolic
backgrounds.

\end{document}